\documentclass[10pt]{amsart}

\usepackage[utf8]{inputenc}
\usepackage[T1]{fontenc}
\usepackage[english]{babel}
\usepackage{amsmath,amssymb,amsfonts,amsthm}
\usepackage{amsmath,amsthm,amsfonts,amssymb, hyperref}
\usepackage[pdftex]{graphicx}
\usepackage{enumitem}
\usepackage{color}
\usepackage{anyfontsize}
\usepackage{caption}

\newtheorem{theorem}{Theorem}[section]

\newtheorem{proposition}[theorem]{Proposition}
\newtheorem{corollary}[theorem]{Corollary}
\newtheorem{remark}[theorem]{Remark}
\newtheorem{lemma}[theorem]{Lemma}

\theoremstyle{definition}
\newtheorem{definition}[theorem]{Definition}

\newcommand{\mD}{\mathcal D}

\begin{document}

\title[Maximal metrics for the first Steklov eigenvalue]{Maximal metrics for the first Steklov eigenvalue on surfaces}

\begin{abstract}
In recent years, eigenvalue optimization problems have received a lot of attention, in particular, due to their connection with the theory of minimal surfaces. In the present paper we prove that on any orientable surface there exists a smooth metric maximizing the first normalized Steklov eigenvalue. For surfaces of genus zero, this has been earlier proved by A.~Fraser and R.~Schoen. Our approach builds upon their ideas and further developments due to Petrides. As a corollary, we show that there exist free boundary branched minimal immersions of an arbitrary compact orientable surface with boundary into a Euclidean ball of some dimension. 
\end{abstract}

\author[Mikhail Karpukhin]{Mikhail Karpukhin}
\address{Department of Mathematics and Statistics, McGill University, Burnside Hall,
805 Sherbrooke Street West,
Montreal, Quebec,
Canada,
H3A 0B9}
\email{mikhail.karpukhin@mail.mcgill.ca}
\maketitle
\section{Introduction}
\subsection{Steklov problem on surfaces}
Let $(M,g)$ be a connected compact Riemannian surface with smooth boundary.
Steklov eigenvalues are defined as numbers $\sigma$ such that there exists a non-zero solution of the following system,
\begin{equation*}
\left\{
   \begin{array}{rcl}
	\Delta u &=& 0\quad \mathrm{on}\,\, M,\\
	\partial_nu &=& \sigma u\quad \mathrm{on}\,\, \partial M,
   \end{array}
\right.
\end{equation*}
where $n$ denotes the outward unit normal vector on $\partial M$.
Steklov eigenvalues coincide with the eigenvalues of the Dirichlet-to-Neumann operator $\mD_g\colon C^\infty(\partial M)\to C^\infty(\partial M)$ defined in the following way. For a function $u\in C^\infty(\partial M)$ let $\hat u$ be the harmonic extension of $u$ to the interior of $M$. Then one sets $\mD_g(u) = \partial_n\hat u$. Operator $\mD_g$ is an elliptic self-adjoint pseudodifferential operator of order $1$.
The spectrum of $\mD_g$ is discrete, its eigenvalues form a sequence
$$
0=\sigma_0(M,g) < \sigma_1(M,g)\leqslant\sigma_2(M,g)\leqslant\sigma_3(M,g)\leqslant\ldots,
$$
where eigenvalues are written with multiplicities. Spectral geometry of the Dirichlet-to-Neumann map has attracted a lot of attention in recent years, see the survey paper~\cite{GPsurvey} and references therein.

For any $k>0$ we introduce the normalized eigenvalue functionals
$$
\bar\sigma_k(M,g) = \sigma_k(M,g) L_g(\partial M),
$$
where $L_g(\partial M)$ stands for the length of the boundary. The functionals $\bar\sigma_k$ are scale-invariant, i.e. for any positive constant $t>0$ one has $\bar\sigma_k(M,tg) = \bar\sigma_k(M,g)$. In the present article we are concerned with extremal problems for the first non-trivial eigenvalue $\bar\sigma_1$. To this end, set
$$
\sigma^*(\gamma,k) =  \sup_g \bar\sigma_1(\Sigma_{\gamma,k},g),
$$
where $\Sigma_{\gamma,k}$ stands for an orientable surface of genus $\gamma$ with $k$ boundary components. The value $\sigma^*(\gamma,k)$ is known to be finite and bounded by a linear function of $\gamma$ and $k$, see~\cite{FS0, GPHPS,KSteklov}. The same is true for higher eigenvalues, see~\cite{KSteklov} for the best known bound.
The problem of geometric eigenvalue optimization is to determine the exact value of $\sigma^*(\gamma,k)$ and to find the metrics for which it is attained.
\begin{definition}
A metric $g$ on $\Sigma_{\gamma,k}$ is called {\em maximal} if $\sigma^*(\gamma,k) = \bar\sigma_1(\Sigma_{\gamma,k},g)$.
\end{definition}

The first result in this direction is due to R.~Weinstock~\cite{W}, who proved that $\sigma^*(0,1)=2\pi$ and the equality is attained on a flat disk. The value $\sigma^*(0,2)$ has not been determined until the seminal work of A.~Fraser and R.~Schoen~\cite{FS1}. Building on their earlier work~\cite{FS0} they studied the question of existence of maximal metrics and were able to resolve it for $\gamma=0$.
For a given $\gamma$ and $k$ the existence of a maximal metric is a delicate question which is closely connected with the theory of {\em free boundary minimal surfaces} in Euclidean balls. 

Recall that an immersion $f\colon M\looparrowright \mathcal{B}^{n}\subset\mathbb{R}^{n}$ is called {\em free boundary minimal immersion} if it is proper, i.e. $f(M)\cap\partial\mathbb{B}^{n} = f(\partial M)$, and critical for the the volume functional in the class of proper immersions. The latter condition is equivalent to vanishing of the mean curvature vector and orthogonality of $f(M)$ to $\partial \mathbb{B}^{n}$.

According to~\cite[Proposition 5.2]{FS1} any maximal metric gives rise to a conformal free boundary immersion into a Euclidean ball $\mathbb{B}^{n}$ (possibly with branched points) such that coordinate functions are first eigenfunctions.
Using this correspondence, in the paper~\cite{FS1} the authors constructed free boundary minimal surfaces in $\mathbb{B}^3$ of genus zero with arbitrary number of boundary components.

%

\subsection{Main result}
\label{MainResults}
Our main result is a generalization of the results of~\cite{FS1} to the case of arbitrary genus. In particular, we prove the following theorem.
\begin{theorem}
\label{MainCorollary}
For all $\gamma\geqslant 0, k\geqslant 1$ the value $\sigma^*(\gamma,k)$ is achieved by a smooth metric. In particular, there exists a free boundary minimal branched immersion $f\colon \Sigma_{\gamma,k}\to\mathbb{B}^{n_{\gamma,k}}$.
\end{theorem}

Theorem~\ref{MainCorollary} provides the first examples of free boundary minimal immersions of $\Sigma_{\gamma,k}$ to the unit ball for arbitrary $\gamma$ and $k$. The examples recently constructed in~\cite{FPZ,FS1,KL,KW,Ketover} are embedded free boundary minimal surfaces in $\mathbb{B}^3$ with small $k$ and sufficiently large $\gamma$ or vice versa. In contrast, the surfaces obtained in Theorem~\ref{MainCorollary} are not necessarily embedded and the dimension $n_{\gamma,k}$ of the ball is not guaranteed to be equal to $3$ but rather bounded by the multiplicity of the first Steklov eigenvalue. 

By multiplicity bounds in~\cite{Jammes2,KKP}, one has the following inequalities,
$$
n_{\gamma,k}\leqslant 4\gamma + 2k.
$$
Moreover, if $\gamma\geqslant 2$ then
$$
n_{\gamma,k}\leqslant 2\gamma + 3.
$$
For discussion on optimality of these bounds and more results for particular values of $\gamma$ and $k$ see~\cite{Jammes2}.

\subsection{Ideas of the proof}  
 
 We make use of regularity theory for maximal metrics developed in~\cite{FS1}. In particular, they proved the case $\gamma=0$ of the following theorem. It was extended to the case $\gamma>0$ by R.~Petrides.

\begin{theorem}[A.~Fraser, R.~Schoen~\cite{FS1}, R.~Petrides~\cite{Petrides}]
\label{regularity}
Suppose that 
\begin{equation}
\label{condition}
\sigma^*(\gamma, k)>\max(\sigma^*(\gamma-1,k),\sigma^*(\gamma,k-1)),
\end{equation}
 where we put $\sigma^*(-1,k) = \sigma^*(\gamma, 0) = -\infty$. Then $\sigma^*(\gamma, k)$ is achieved by a smooth metric.
\end{theorem} 

In the same paper~\cite{FS1} the authors proved a part of the condition~\eqref{condition}.

\begin{theorem}[A.~Fraser, R.~Schoen~\cite{FS1}]
\label{hole}
Suppose that $\sigma^*(\gamma, k)$ is achieved by a smooth metric. Then $\sigma^*(\gamma, k+1)>\sigma^*(\gamma, k)$.
\end{theorem}

Note that if $\gamma = 0$, then condition~\eqref{condition} becomes simply $\sigma^*(0, k+1)>\sigma^*(0, k)$ which is exactly a particular case of Theorem~\ref{hole}. This is the exact observation that allowed A.~Fraser and R.~Schoen to conclude the existence of free boundary minimal surfaces of genus zero.  

Prior to the work of A.~Fraser and R.~Schoen~\cite{FS1} only two examples of free boundary minimal surfaces in $\mathbb{B}^n$ were known: flat disk and so-called {\em critical catenoid}, a unique catenoid that intersects the unit sphere orthogonally. As it was proved in~\cite{FS1}, critical catenoid is the unique maximal metric on an annulus.  Nowadays there are several other constructions of free boundary minimal surfaces available: by gluing techniques~\cite{FPZ,KL,KW} and by min-max method~\cite{Ketover}. It is unknown whether the metrics induced on these minimal surfaces are maximal.

In the present article we complete the proof of condition~\eqref{condition}. In particular, in Section~\ref{construction} we prove the following more general theorem.

\begin{theorem}
\label{MainTheorem}
For all $\gamma\geqslant 0, k\geqslant 1$ one has $\sigma^*(\gamma+1,k)\geqslant \sigma^*(\gamma,k+1)$
\end{theorem}

As a corollary, we obtain the second part of condition~\eqref{condition} complementing Theorem~\ref{hole}.
\begin{corollary}
\label{handle}
Suppose that $\sigma^*(\gamma, k)$ is achieved by a smooth metric. Then $\sigma^*(\gamma+1, k)>\sigma^*(\gamma, k)$.
\end{corollary}
\begin{proof}
By Theorem~\ref{hole} and Theorem~\ref{MainTheorem} we have a chain of inequalities
$$
\sigma^*(\gamma+1,k)\geqslant \sigma^*(\gamma,k+1) > \sigma^*(\gamma,k).
$$
\end{proof}

\begin{proof}[Idea of the proof of Theorem~\ref{MainTheorem}] Given a surface $M$ of genus $\gamma$ with $k+1$ boundary components we take two segments of equal length $\varepsilon$ lying on different boundary components. We then glue these segments together to obtain a new surface $M_\varepsilon$ of genus $\gamma+1$ with $k$ boundary components (see Figure 1). The main technical step is to prove that $\bar\sigma_1(M_\varepsilon)\to\bar\sigma_1(M)$ as $\varepsilon\to 0$.
\end{proof}

\begin{proof}[Proof of Theorem~\ref{MainCorollary}]
The proof is by induction on $\gamma+k$. The base $\gamma = 0$, $k=1$ is covered by the Weinstock result for $\sigma^*(0,1)$. Suppose it is proved for all $\gamma,k$ such that $\gamma + k = n$, let us prove it for all $\gamma' + k' = n+1$. By Theorem~\ref{regularity} it is sufficient to show that $\sigma^*(\gamma',k')>\max(\sigma^*(\gamma'-1,k'),\sigma^*(\gamma',k'-1))$. The latter follows from the induction step, Theorem~\ref{hole} and Corollary~\ref{handle}.
\end{proof}

Finally, let us remark that the theory of maximal metrics for the first Steklov eigenvalue is parallel to the theory of maximal metric for the first Laplace eigenvalue, see e.g.~\cite{KNPP,NP, Petrides}. The paper~\cite{KNPP} contains a survey of the recent results. The analog of Corollary~\ref{MainCorollary} has been recently obtained by H.~Matthiesen and A.~Siffert~\cite{MS}. They verify the monotonicity condition similar to~\eqref{condition} which was obtained by R.~Petrides in~\cite{Petrides0}.

\subsection{Organization of the paper} In Section~\ref{prelim} we collect some preliminary facts including discussions on density formulation for Steklov eigenvalues in Section~\ref{prelim1}, non-smooth densities in Section~\ref{prelim2} and conical singularities in Section~\ref{prelim3}. Section~\ref{sectionconvergence} contains a rather general convergence result for the first Steklov eigenvalue. Finally, we prove Theorem~\ref{MainTheorem} in Section~\ref{construction}.

\section{Preliminaries}
\label{prelim}
\subsection{Steklov problem and densities}
\label{prelim1}
In the case of surfaces there is an equivalent but slightly different formulation of the Steklov problem which highlights conformal invariance of harmonic extension. Let $g$ be a metric on $M$ and let $\rho>0$ be a smooth density function. Then we define $\mD_{g,\rho}(u) = \frac{1}{\rho}\partial_n\hat u$. The harmonic extensions of the eigenfunctions of this operator are solutions of
\begin{equation}
\label{rhoSteklov}
\left\{
   \begin{array}{rcl}
	\Delta u &=& 0\quad \mathrm{on}\,\, M,\\
	\partial_nu &=& \sigma\rho u\quad \mathrm{on}\,\, \partial M.
   \end{array}
\right.
\end{equation}

Similarly to $\mD_g$ the operator $\mD_{g,\rho}$ possesses purely discrete spectrum that forms a sequence
$$
0=\sigma_0(M,g,\rho) < \sigma_1(M,g,\rho)\leqslant\sigma_2(M,g,\rho)\leqslant\sigma_3(M,g,\rho)\leqslant\ldots,
$$
where eigenvalues are written with multiplicities. We define the normalized eigenvalues 
$$
\bar\sigma_k(M,g,\rho) = \sigma_k(M,g,\rho) L_{g,\rho}(\partial M), 
$$
where $L_{g,\rho}(\partial M) = \int\limits_{\partial M}\rho\,d\sigma_g$ is the weighted length of the boundary. Thus, we recover $\bar\sigma_k(M,g)$ by setting $\rho\equiv 1$.

Let $\tilde\rho$ be any positive extension of $\rho$ to the interior of $M$. Then by conformal invariance of the Laplacian and the fact that $\rho^2g(\frac{1}{\rho}n,\frac{1}{\rho}n) = g(n,n) = 1$ one has $\mD_{g,\rho} = \mD_{\tilde\rho^2 g}$ and $\bar\sigma_k(M,g,\rho) = \bar\sigma_k(M,\tilde\rho^2 g)$.
Similarly, for any positive function $f\in C^\infty(M)$ one has $\mD_{f^2g,\rho} = \mD_{g,f\rho}$ and
\begin{equation}
\label{conformal}
\bar\sigma_k(M,f^2g,\rho) = \bar\sigma_k(M,g,f\rho).
\end{equation}

We see that positive smooth densities can be interpreted as conformal factors. In the following section we define eigenvalues for non-negative $\rho\in L^\infty(\partial M)$.


\subsection{Non-smooth densities}
\label{prelim2}
In this section we follow the paper of Kokarev~\cite{Kokarev} where a comprehensive study of eigenvalue optimization in the context of measure theory has been done. The results of Kokarev are much more general than we require, so we recall only basic statements.

Given a compact surface (possibly with boundary), a conformal class $c$ of metrics on $M$ and a non-zero measure $\mu$, the eigenvalues $\lambda_k(\mu,c)$ are defined as critical values of the Rayleigh quotient
\begin{equation}
\label{Rayleigh}
R_c(\mu,u) = \frac{\displaystyle\int\limits_M |\nabla u|_g^2\,dv_g}{\displaystyle\int\limits_M u^2\,d\mu},
\end{equation}
in the appropriate space of functions, where $g$ is any metric in the conformal class $c$. For an arbitrary measure $\mu$ the existence of eigenfunctions is a delicate question, see e.g.~\cite[Proposition 1.3, Lemma~2.7]{Kokarev}. However, we will apply this definition to the case $\mu = \rho\,d\sigma_g$, where $\rho\in L^\infty(d\sigma_g)$ is a non-negative function. In this case the existence of eigenfunctions easily follows from the trace embedding theorem and we use $\sigma_k(M,g,\rho)$ to denote $\lambda_k(\rho\,d\sigma_g,c)$. Since a constant function is always an eigenfunction corresponding to the zero eigenvalue, one has
$$
\sigma_1(M,g,\rho) = \min\limits R_c(\rho\,d\sigma_g,u),
$$
where $u$ ranges over all non-zero functions in the Sobolev space $H^1(M,g)$ satisfying $\int\limits_{\partial M}u\rho\,d\sigma_g = 0$.

The eigenfunctions corresponding to the eigenvalue $\sigma$ are precisely solutions to~\eqref{rhoSteklov} with non-negative $\rho\in L^\infty(\partial M)$. Classical elliptic regularity estimates for the operator $\mD_g$ yield that the eigenfunctions are at least $\frac{1}{2}$-H\"older regular on the boundary, i.e. they lie in $C^{\frac{1}{2}}(\partial M)$, see Remark~\ref{Continuity} below.

\subsection{Conical singularities}
\label{prelim3}
The construction provided in the sketch of the proof of Theorem~\ref{MainTheorem}(see Section~\ref{MainResults}) yields a smooth metric with two conical singularities of angle $4\pi$ on the boundary. For more detailed review of such metrics see e.g.~\cite{NP}. For our purposes we view a metric $h$ with conical singularities of angle $2\pi k$, $k\in\mathbb{N}$ via an equality $h = f^2 h_0$, where $h_0$ is a regular smooth Riemannian metric and $f\geqslant 0$ is a smooth function that vanishes exactly at singular points. The order of vanishing at $p\in M$ equals $k-1$ iff the corresponding conical angle at $p$ equals $2\pi k$.

A smooth metric $h$ with conical singularities induces well-defined measures $dv_h = f^2dv_{h_0}$, $d\sigma_h = fd\sigma_{h_0}$ and the Dirichlet integral with respect to the metric $h$ in the numerator of~\eqref{Rayleigh} is defined as an improper integral. The latter is known to coincide with the Dirichlet integral with respect to the metric $h_0$, see e.g.~\cite[Example 1.1]{Kokarev}. As a result, the definition of Section~\ref{prelim2} $\sigma_k(M,h,\rho) = \lambda_k(\rho\,d\sigma_h,[h])$ can still be applied for such metrics $h$.
Moreover, the previous discussion yields the equality
\begin{equation}
\label{conformal2}
\bar\sigma_k(M,h,\rho)  = \bar\sigma_k(M,h_0,f\rho),
\end{equation}
that can be seen as an extension of~\eqref{conformal} to the case when the function $f$ can have isolated zeroes.

\section{Convergence of the first eigenvalues}
\label{sectionconvergence}
It turns out that allowing non-smooth densities does not affect the value of $\sigma^*(\gamma,k)$. In order to see that we prove the following lemma.

\begin{lemma}
\label{convergence}
Let $\rho_\varepsilon\in L^\infty(d\sigma_g)$ be a uniformly bounded family of non-negative functions, such that $||\rho_\varepsilon||_{L^1(d\sigma_g)}=1$ and $\rho_\varepsilon\to \rho$ in $L^1(d\sigma_g)$. Then 
$$
\lim\limits_{\varepsilon\to 0} \sigma_1(M,g,\rho_\varepsilon) = \sigma_1(M,g,\rho).
$$
\end{lemma}

To motivate the statement let us deduce several corollaries of Lemma~\ref{convergence}.

\begin{proposition}
\label{Linfty}
One has 
$$
\sigma^*(\gamma,k) = \sup
\sigma_1(M,g,\rho)L_{g,\rho}(\partial M),
$$
where the supremum is taken over all smooth metrics $g$ and all non-negative densities $\rho\in L^\infty(\partial M)\backslash\{0\}$.
\end{proposition} 
\begin{proof}
By Lemma~\ref{convergence} it suffices to approximate any non-negative $\rho\in L^\infty$ by positive smooth functions in the $L^1(d\sigma_g)$-norm. It can be easily done by making a canonical regularization procedure using the heat kernel $K_t(x,y)$ of the boundary, see~\cite{FS1,Petrides}. We set 
$$
\rho_\varepsilon (y) = \int\limits_{\partial M} K_\varepsilon(x,y)\rho\,d\sigma_g(x).
$$ 
As $K_\varepsilon(x,y)$ is a smooth positive function, one has $\rho_\varepsilon$ is smooth and positive. Moreover, $K_\varepsilon(x,y) \to \delta_{\{x=y\}}$, therefore $\rho_\varepsilon\to\rho$ as $\varepsilon\to 0$.
\end{proof}

\begin{corollary}
\label{conical}
Suppose that $g$ is a smooth metric with isolated conical singularities and $\rho\in L^\infty(\partial M)$ is a positive density. Then there exists a sequence of smooth metrics $g_i$ such that
$$
\lim\limits_{i\to\infty}\sigma_1(M,g_i)L_{g_i}(\partial M)\to \sigma_1(M,g,\rho) L_{g,\rho} (\partial M).
$$
Moreover, metrics $g_i$ can be chosen conformal to $g$.
\end{corollary}
\begin{proof}
The statement follows from equality~\eqref{conformal2} and Proposition~\ref{Linfty}.
\end{proof}

\begin{proof}[Proof of Lemma~\ref{convergence}]
Let $\sigma_{\varepsilon} = \sigma_1(M,g,\rho_\varepsilon)$ and let $f_\varepsilon$ be a corresponding Steklov eigenfunction normalized so that $||f_\varepsilon||_{L^2(\rho_\varepsilon d\sigma_g)} = 1$. Similarly, let $\sigma = \sigma_1(M,g,\rho)$ and $f$ is a corresponding eigenfunction
such that $||f||_{L^2(\rho d\sigma_g)} = 1$. For the remaining of the proof we adapt the convention that $C$ denotes all the constants depending on $M$ and $g$ but not on $\varepsilon$.

{\bf Claim 1.}
There exists $\varepsilon_0>0$ such that the eigenvalues $\sigma_\varepsilon$ are uniformly bounded for $\varepsilon<\varepsilon_0$.

 We use $f$ as a test function for $\sigma_\varepsilon$. In order to do that we need to modify it so that it is orthogonal to constants in $L^2(\rho_\varepsilon d\sigma_g)$, i.e. we use 
 $$
 f - \int\limits_{\partial M}f\rho_\varepsilon\,d\sigma_g
 $$
 as a test function. One obtains 
 \begin{equation}
 \label{sigmaeps}
 \sigma_\varepsilon \leqslant\cfrac{\displaystyle\int\limits_M|\nabla f|^2\,dV_g}{\displaystyle\int\limits_{\partial M}f^2\rho_\varepsilon\,d\sigma_g - \left(\displaystyle\int\limits_{\partial M} f\rho_\varepsilon\,d\sigma_g\right)^2}.
 \end{equation}
 As $f$ is bounded (see Remark~\ref{Continuity} below), the conditions of the lemma imply 
 that $f^2\rho_\varepsilon\to f^2\rho$ and $f\rho_\varepsilon\to f\rho$ in $L^1(d\sigma_g)$. Since one has $\int\limits_{\partial M}f\rho\,d\sigma_g = 0$, taking the limit in~\eqref{sigmaeps} yields
 \begin{equation}
 \label{limit}
 \limsup\limits_{\varepsilon\to 0}\sigma_\varepsilon\leqslant \sigma. 
 \end{equation}
 The claim follows.
 
{\bf Claim 2.} There exists $\varepsilon_0>0$ such that for all $\varepsilon<\varepsilon_0$ the norms $||f_\varepsilon||_{L^2(d\sigma_g)}$ and $||{f_\varepsilon}||_{H^1(dv_g)}$ are uniformly bounded.
 
In order to prove this claim let us recall the following generalized Poincar\'e inequality.
\begin{theorem}[\cite{AH} Lemma 8.3.1] 
\label{Poincare}
Let $(M,g)$ be a Riemannian manifold. Then there exists a constant $C>0$ such that for all $L\in H^{-1}(M)$ with $L(1) = 1$ one has 
\begin{equation}
||u- L(u)||_{L^2(M)} \leqslant C||L||_{H^{-1}(M)} \left(\,\int\limits_M |\nabla u|^2_g\,dv_g\right)^{1/2}
\end{equation}
for all $u\in H^1(M)$.
\end{theorem}
 
 We apply this theorem for $L_\varepsilon(u) = \int\limits_{\partial M} u\rho_\varepsilon\,d\sigma_g$. First, let us compute their norm,
 $$
 \left|\,\int\limits_{\partial M} u\rho_\varepsilon\, d\sigma_g \right|\leqslant C\int\limits_{\partial M}|u|\,d\sigma_g \leqslant C||u||_{L^2(d 
 \sigma_g)}\leqslant C||u||_{H^1(M,g)},
 $$ 
 where we used in order: boundedness of $\rho_\varepsilon$, Cauchy-Schwarz and trace inequalities. Thus, we see that $L_\varepsilon$ is a uniformly bounded family in $H^{-1}(M)$. According to Theorem~\ref{Poincare} applied to $L_\varepsilon$ and the eigenfunction $f_{\varepsilon}$, one has
 $$
 ||{f_\varepsilon}||_{L^2(dv_g)}\leqslant C\left(\,\int\limits_M |\nabla  {f_\varepsilon}|^2\,dv_g\right)^{1/2} = C\sqrt{\sigma_\varepsilon},
 $$
 since $L_\varepsilon( f_\varepsilon) = 0$.
 The right hand side is uniformly bounded by Claim 1. Thus, we obtain that $f_\varepsilon$ are uniformly bounded in $H^1(dv_g)$. Trace embedding theorem implies that $f_\varepsilon$ are uniformly bounded in $L^2(d\sigma_g)$.
 
{\bf Claim 3.}
There exists $\varepsilon_0>0$ such that the functions $f_\varepsilon$ are uniformly bounded for $\varepsilon<\varepsilon_0$.

Let $\mD_g$ denote the Dirichlet-to-Neumann operator associated with the metric $g$. Then
the function $f_\varepsilon$ satisfies $\mD_gf_\varepsilon = \sigma_\varepsilon \rho_\varepsilon f_\varepsilon$.
As $M$ is a manifold with smooth boundary, $\mD_g$ is an elliptic self-adjoint $\Psi$DO of order $1$. Therefore, by elliptic regularity and $1$-dimensional Sobolev embedding theorem, one has
$$
||f_\varepsilon||_{\infty}\leqslant C||f_\varepsilon||_{H^1(d\sigma_g)}\leqslant C(||\sigma_\varepsilon\rho_\varepsilon f_\varepsilon||_{L^2(d\sigma_g)} + ||f_\varepsilon||_{L^2(d\sigma_g)})\leqslant C(\sigma_\varepsilon + 1),
$$
where in the last inequality we used the boundedness of $\rho_\varepsilon$ and Claim 2.
According to Claim 1 the right hand side is uniformly bounded.

\begin{remark}
\label{Continuity}
The same argument applied to a single Steklov eigenfunction corresponding to $\sigma_k(M,g,\rho)$, combined with an embedding $H^1(\partial M)\to C^{\frac{1}{2}}(\partial M)$, shows that all eigenfunctions belong to $C^{\frac{1}{2}}(\partial M)$.
\end{remark}

{\bf Claim 4.} One has 
$$
\int\limits_{\partial M} f^2_\varepsilon \rho\,d\sigma_g \to 1 \qquad \mathrm{and}\qquad \int\limits_{\partial M} f_\varepsilon\rho\, d\sigma_g \to 0
$$
as $\varepsilon \to 0$.

Indeed, since $||f_\varepsilon||_{L^2(\rho_\varepsilon\,d\sigma_g)} = 1$ one has
$$
\left|\,\int\limits_{\partial M} f^2_\varepsilon \rho\,d\sigma_g - 1\right|\leqslant \int\limits_{\partial M} |f_\varepsilon|^2|\rho-\rho_\varepsilon|\,d\sigma_g.
$$
Claim 3 and $L^1$-convergence $\rho_\varepsilon\to\rho$ conclude the proof. The second limit is proved in the same way, since the mean of $f_\varepsilon$ with respect to $\rho_\varepsilon\,d\sigma_g$ is equal to $0$.

To conclude the proof of Lemma~\ref{convergence}, we use $f_\varepsilon$ as test-functions for $\sigma$.  
We modify $f_\varepsilon$ so that it is orthogonal to constants in $L^2(\partial M)$, i.e. we use
$$
f_\varepsilon - \int\limits_{\partial M}f_\varepsilon\rho_\varepsilon\,d\sigma_g
$$
as a test function. Thus, similarly to~\eqref{sigmaeps}, we obtain
\begin{equation*}
 \sigma \leqslant\cfrac{\displaystyle\int\limits_M|\nabla  f_\varepsilon|^2\,dV_g}{\displaystyle\int\limits_{\partial M}f_\varepsilon^2\rho\,d\sigma_g - \left(\displaystyle\int\limits_{\partial M} f_\varepsilon\rho\,d\sigma_g\right)^2},
 \end{equation*}
where the numerator of the fraction on the right hand side equals $\sigma_\varepsilon$ and the denominator tends to $1$ by Claim 4. Therefore, we obtain
$$
\sigma\leqslant\liminf\limits_{\varepsilon\to 0}\sigma_\varepsilon.
$$
In combination with~\eqref{limit} this concludes the proof of the lemma.
\end{proof}


%

\section{Proof of Theorem~\ref{MainTheorem}}

\label{construction}
The construction in this section is reminiscent of the construction in~\cite[Lemma 2.1.2]{GPHPS}, where the authors are gluing together two plane disks.
Let $g$ be a smooth metric on $M = \Sigma_{\gamma,k+1}$ and $\rho>0$ be a smooth function such that $L_{g,\rho}(\partial M) = 1$ and $\sigma_1(M,g,\rho)$ is close to $\sigma^*(\gamma,k+1)$. 
Let $p_1,p_2\in\partial M$ be two points lying on different boundary components of $M$. By conformal invariance we can assume, without loss of generality, that there exist neighbourhoods $U_i$ of $p_i$ such that the restriction $g|_{U_i}$ is a Euclidean metric.
For all sufficiently small $\varepsilon>0$ and $j=1,2$ let $I_{j,\varepsilon}\subset\partial M$ be a segment of the boundary of length $2\varepsilon$ centered at $p_i$, where the length is taken with respect to the metric $g$. Note that there is a unique orientation reversing and distance preserving map $\varphi_\varepsilon\colon I_{1,\varepsilon}\to I_{2,\varepsilon}$ satisfying $\varphi_\varepsilon(p_1) = p_2$.

We define a new surface $M_\varepsilon$ by gluing together $I_{1,\varepsilon}$ and $I_{2,\varepsilon}$ using $\varphi_\varepsilon$, i.e. $M_\varepsilon = M/ (p\sim \varphi_\varepsilon(p))$, see Figure 1. Let $I_\varepsilon\subset M_{\varepsilon}$ denote the common image of $I_{1,\varepsilon}$ and $I_{2,\varepsilon}$ in $M_\varepsilon$. 


\begin{figure}
  \centering
  \def\svgwidth{\columnwidth}
  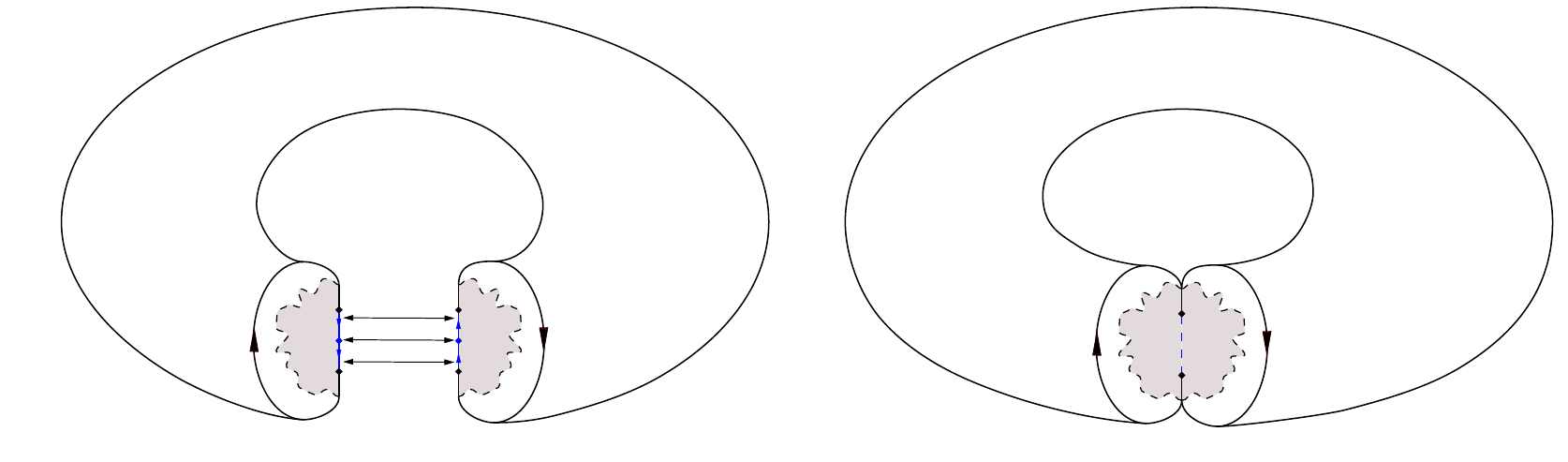
  \begin{minipage}{127mm}
    \footnotesize
    \emph{
    Figure 1: On the left, we see the surface $M$ together with the boundary orientation. The identification $\varphi_\varepsilon$ is denoted by horizontal arrows. On the right, we see $M_\varepsilon$. Note that at the endpoints of $I_\varepsilon$, the surface is locally isometric to a neighbourhood of $0$ in $\mathbb{R}^2\backslash \{(x,0),\,x>0\}$, i.e. the metric $g_\varepsilon$ has conical singularities of angle $4\pi$ at these endpoints.
    }
   \end{minipage}
\end{figure}

Since the metric $g$ is Euclidean in $U_i$, a neighbourhood of either endpoint of $I_\varepsilon$ is locally isometric to a Euclidean angle $2\pi$ that can be represented by $\mathbb{R}^2\backslash \{(x,0),\,x>0\}$. Therefore,
one has the following proposition.

\begin{proposition}
For all $\varepsilon>0$ the surface $M_\varepsilon$ is an orientable surface of genus $\gamma+1$ with $k$ boundary components. The metric $g$ induces a smooth metric $g_\varepsilon$ on $M_\varepsilon$ with two conical singularities of angle $4\pi$ on the boundary located at the endpoints of $I_\varepsilon$.
\end{proposition}

The smooth function $\rho>0$ on $\partial M$ naturally descends to an $L^\infty$ function $\rho_\varepsilon>0$ on $\partial M_\varepsilon$ smooth away from conical singularities such that $L_{g,\rho_\varepsilon}(\partial M_\varepsilon)\to 1$ as $\varepsilon$ tends to $0$. 

\begin{proposition}
\label{limit2}
One has
$$
\liminf\limits_{\varepsilon\to 0}\sigma_1(M_\varepsilon,g_\varepsilon,\rho_\varepsilon) \geqslant \sigma_1(M,g,\rho).
$$
\end{proposition}
\begin{proof}
A classical bracketing argument, see e.g.~\cite{Davies}, yields that introducing additional Neumann conditions does not increase the spectrum. Let us apply this argument to $I_\varepsilon\subset M_{\varepsilon}$, the common image of $I_{1,\varepsilon}$ and $I_{2,\varepsilon}$ in $M_\varepsilon$. Since Neumann conditions can be viewed as Steklov conditions with density equal to zero, one obtains
\begin{equation}
\label{ineq}
\sigma_1(M,g,\tilde\rho_\varepsilon)\leqslant\sigma_1(M_\varepsilon, g_\varepsilon,\rho_\varepsilon),
\end{equation}
where 
\begin{equation*}
\tilde\rho_\varepsilon = 
\begin{cases}
0 \quad\mathrm{on} &\quad  I_{1,\varepsilon}\cup I_{2,\varepsilon}\\
\rho\quad\mathrm{otherwise}.
\end{cases}
\end{equation*}
As $\tilde\rho_\varepsilon\to\rho$ in $L^1(d\sigma_g)$ an application of Lemma~\ref{convergence} yields that the left hand side of~\eqref{ineq} tends to $\sigma_1(M,g,\rho)$. Therefore, taking $\liminf$ in~\eqref{ineq} concludes the proof.
\end{proof}

As $g$ and $\rho$ can be chosen arbitrarily close to $\sigma^*(\gamma, k+1)$, Proposition~\ref{limit2} and Corollary~\ref{conical} imply Theorem~\ref{MainTheorem}.

%

\subsection*{Acknowledgements} The author is grateful to I. Polterovich, D. Bucur and J. Galkowski for fruitful discussions. 
The author would also like to thank I. Polterovich and V. Medvedev for remarks on the preliminary version of the manuscript.
This research was partially supported by Schulich Fellowship. This work is a part of the author’s PhD thesis at McGill University under the supervision of Dmitry Jakobson and Iosif Polterovich.

\end{document}